\documentclass[12pt,a4paper, reqno]{amsart}
\usepackage[utf8]{inputenc}
\usepackage{mathptmx}
\usepackage{amssymb}
\usepackage{hyperref}
\usepackage{mathrsfs}
\usepackage[mathscr]{euscript}
\usepackage{booktabs}

\usepackage{todonotes}

\usepackage[ignoreunlbld,norefs,nocites,nomsgs]{refcheck}

\usepackage{pgf, tikz}

\usepackage{fancyvrb} \RecustomVerbatimEnvironment{Verbatim}{Verbatim}
{xleftmargin=15pt, frame=single, fontsize=\small}

\usepackage{color}

\definecolor{light}{gray}{0.9}
\definecolor{medium}{gray}{0.8}

\newtheorem{theorem}{Theorem}
\newtheorem{lemma}[theorem]{Lemma}

\newtheorem{proposition}[theorem]{Proposition}
\newtheorem{conjecture}[theorem]{Conjecture}

\theoremstyle{definition}
\newtheorem{remark}[theorem]{Remark}
\newtheorem{definition}[theorem]{Definition}

\numberwithin{equation}{section}

\def\NN{{\mathbb N}}

\def\11{{\mathbb 1}}
\def\cR{{\mathcal R}}

\def\cB{{\mathcal B}}

\def\cF{{\mathcal F}}
\def\cM{{\mathcal M}}

\def\Rees{\operatorname{\cR}}

\def\ttt#1{\texttt{#1}}

\let\epsilon=\varepsilon

\allowdisplaybreaks

\def\strut{\vphantom{\Large(} }

\def\tat{t\^ete-a-t\^ete}
\def\init{\operatorname{init}}
\def\ini{\operatorname{in}}
\def\Ker{\operatorname{Ker}}
\def\supp{\operatorname{supp}}

\def\Sagbi{{\textsc{Sagbi}}}

\begin{document}
	

\title{\Sagbi {} combinatorics of maximal minors and a \Sagbi {} algorithm}

\author{Winfried Bruns}
\address{Universit\"at Osnabr\"uck, Institut f\"ur Mathematik, 49069 Osnabr\"uck, Germany}
\email{wbruns@uos.de}
\author{Aldo Conca}
\address{Dipartimento di Matematica, Universit\'a di Genova, Italy}
\email{conca@dima.unige.it}

\subjclass[2010]{13F50, 13F65, 13P10, 13P99, 14M25}
\keywords{Grasmannian, maximal minors, \Sagbi {} basis, Singular, Normaliz}
\thanks{AC is supported by PRIN 2020355B8Y  and by INdAM-GNSAGA} 

\begin{abstract}
The maximal minors of a matrix of indeterminates are a universal
Gr\"obner basis by a theorem of Bernstein, Sturmfels and Zelevinsky. On the
other hand it is known that they are not always a universal \Sagbi {} basis. By an experimental
approach we discuss their behavior under varying monomial orders and their extensions to
\Sagbi {} bases. These experiments motivated a new implementation of the \Sagbi {} algorithm
which is organized in a Singular script and falls back on Normaliz for the combinatorial
computations. In comparison to packages in the current standard distributions of Macaulay 2, version 1.21, and 
Singular, version 4.2.1 and a package intended for CoCoA 5.4.2, it extends the range of computability by at least one order of magnitude.
\end{abstract}
\maketitle

\section{Introduction}

Let $X$ be an $m\times n$ matrix of indeterminates over a field $K$. The minors of $X$, i.e., the determinants of the submatrices of $X$ generate subalgebras and ideals of the polynomial ring $R=K[X_{ij}: i=1,\dots,m\ j=1\dots,n]$ that are important not only in algebraic geometry and commutative algebra, but also in representation theory and combinatorics. For a recent survey of the manifold approaches to the theory of ideals and subalgebras generated by minors we refer the reader to our recent volume \cite{BCRV} with Raicu and Varbaro. One of these approaches is via Gr\"obner and \Sagbi {} bases. We trust that the reader is familiar with the notion of Gr\"obner basis. A \Sagbi {} basis (terminology of Robbiano and Sweedler  \cite{RobSwe}) of a subalgebra $A$ for a given monomial (or term) order on the ambient polynomial ring is a system of generators $\cF$  such that the set $\ini(\cF)$ of initial monomials generates the initial algebra $\ini(A)$.

Let us assume $m\le n$. By a theorem of Sturmfels \cite{StuGrDet}, the $t$-minors, i.e., the determinants of the $t\times t$-submatrices, are a Gr\"obner basis of the ideal they generate under any diagonal monomial order---a diagonal monomial order chooses the product of the diagonal elements as the initial monomial of a $t$-minor. For the maximal minors, for which $t=m$, much more is true: They are a universal Gr\"obner basis, namely a Gr\"obner basis for an arbitrary monomial order. See Conca, De Negri and Gorla \cite{CDG} for a simple proof of this theorem of Bernstein, Sturmfels and Zelevinsky \cite{BerZel, StuZel}. 

The  subalgebra $G(m,n)$ generated by the maximal minors is the homogeneous coordinate ring of the Grassmann variety of $m$-dimensional subspaces of $K^n$. For general $t$ the $t$-minors are not a \Sagbi {} basis of the algebra they generate. But for maximal minors this is true for diagonal orders, and the toric deformation of $G(m,n)$ that it offers gives comfortable access to many homological and enumerative invariants of $G(m,n)$.  Other term orders
  for which maximal minors are a \Sagbi {} bases of $G(m,n)$ have been recently discovered, see for example  \cite{CM} and \cite{HO}.  

In view of the Bernstein--Sturmfels--Zelevinsky theorem the question of universality arises also here. But it was disproved by Speyer and Sturmfels \cite{SpeyStu}: for a $3\times 6$ matrix there exists a lexicographical order under which the $3$-minors are not a \Sagbi {} basis of $G(3,6)$.  

The starting point of the project on which we report in this paper was the question of what can be said about reverse lexicographical orders in this respect. As we will see, revlex universality also fails, but needs a $3\times 8$ matrix for failure. In view of the special properties of normal monomial algebras it is also interesting to ask whether initial algebras of $G(m,n)$ are always normal. This is surprisingly often true, and finding a counterexample needs more patience. By our experimental methods whose results are reported in Section \ref{Grass}, we cannot answer the question whether there is a finite universal  \Sagbi {} basis of $G(m,n)$. But our findings support the conjecture that this is true.

Our investigation of \Sagbi {} bases of Grassmannians motivated the development of an algorithm for the computation of \Sagbi {} bases in general. It is almost impossible to devise a  really new algorithm, but the efficiency of the implementation is extremely important for this complicated task. Our algorithm is implemented in the Singular \cite{Sing} script sagbiNormaliz.lib. While Singular is used for the polynomial arithmetic, newly developed functions of Normaliz \cite{Nmz} are called for the combinatorial tasks. The variants of our algorithms are explained in Section \ref{Impl}. One of them is based on control by Hilbert series.

In the last section we compare our algorithm to packages that come with the standard distributions of Macaulay2\cite{M2}, version 1.21, and Singular \cite{Sing}, version 4.2.1. We also compare it with a preliminary version of a CoCoA5 \cite{CoCoA} package which will be in the standard distribution from version 5.4.2 on. Applications to different types of subalgebras show that our implementation extends the range of computability by more than one order of magnitude. 
 
Normaliz 3.10.0 together with the Singular library sagbiNormaliz.lib and a new version of normaliz.lib , has been published in January 2023. On request we will provide the software and all data of this project to interested readers.

\section{Basics of \Sagbi {} bases}\label{basics}

The reader finds a compact discussion of \Sagbi {} bases in \cite[Sect. 1.3]{BCRV}. We use the notation developed there. Kreuzer and Robbiano \cite[Sect. 6.6]{KrRo} give a more extensive introduction; also see Ene and Herzog \cite{EneHerz} and Sturmfels \cite{StuGrPol}. \Sagbi {} bases were introduced independently by Robbiano and Sweedler \cite{RobSwe} and Kapur and Madlener \cite{KapMad}. The acronym \Sagbi {} stands for ``subalgebra analog to Gr\"obner bases of ideals'' \cite{RobSwe}.     Some authors have adopted recently a new terminology, Khovanskii bases, for a notion that generalize that \Sagbi {} bases but  in this paper we keep the traditional name.

Let $A \subset R=K[X_1,\dots,X_n]$ be a $K$-subalgebra and $\cF$ a (not necessarily finite) family of polynomials belonging to $A$. We assume that $R$ is endowed with a monomial (or term) order $<$. In the following we will often simply speak of an \emph{order}. One calls $\cF$ a \emph{\Sagbi {} basis} of $A$ if the initial monomials $\ini(f)$, $f\in \cF$, generate the initial algebra $\ini(A)$. A \Sagbi {} basis is automatically a system of generators of $A$. If $\cF$ is finite, then $A$ and $K[\ini(\cF)]$ are connected by a flat deformation (Conca, Herzog and Valla \cite{CHV}), and this allows the transfer homological and enumerative data from the toric algebra $K[\ini(\cF)]$ to $A$. Chapter 6 of \cite{BCRV} exploits this approach for the investigation of algebras generated by minors. 

\Sagbi {} bases need not be finite, but can always be chosen countable. Therefore one must allow that $\cF= (f_u)_{u\in N}$ with $N=\{1,\dots,p\}$ or $N=\NN$. We will always assume that the members of $\cF$ are monic. This is evidently no essential restriction of generality as long as the base ring $K$ is a field.

For us the following simple lemma is an important tool in the computation of \Sagbi {} bases. For a $\NN$-graded vector $K$-vector space $V$ one defines the \emph{Hilbert function} of $V$ by
\begin{equation}
H(V,k) = \dim_K V_k, \qquad k\in\NN,\label{Hilb}
\end{equation}
where $V_k$ is the subspace of degree $k$ elements of $V$.

\begin{lemma}\label{HilbLemma}
Let $R$ be a polynomial ring, endowed with a monomial order, and $A$ a finitely generated graded subalgebra. Furthermore let $\cF$ be a family of polynomials in $A$ and $B=K[\ini(\cF)]$ the subalgebra generated by the monomials $\ini(f)$, $f\in \cF$.Then the following hold:
\begin{enumerate}
\item $H(B,k) \le H(A,k)$ for all $k$.
\item $\cF$ is a \Sagbi {} basis of $A$ if and only if $H(B,k) = H(A,k)$ for all $k\in\NN$.
\end{enumerate}
\end{lemma}

\begin{proof}
For a graded subspace $V$ of the polynomial ring one has $H(V,k)=H(\ini(V)),k)$ for all $k$ \cite[1.4.3]{BCRV}. Furthermore there are inclusions
$$
B_k \subset \ini(A)_k,\qquad k\in\NN,
$$
and equality holds if and only if $H(B_k,k)=H(\ini(A)_k),k)$ for all $k$. Together with Equation \eqref{Hilb} this proves the lemma.
\end{proof}

To present $A$ as a residue class ring of a polynomial ring, we choose $P=K[Y_u: u\in N]$ and define a surjection
\begin{equation}
\phi: P \to A, \qquad \phi(Y_u) = f_u, \ u\in N.\label{phi}
\end{equation}
The $K$-algebra  $K[\ini(\cF)]$  is a homomorphic image of $P$, as well, namely by the surjection
$$
\psi: P     \to K[\ini(\cF)], \qquad \psi(Y_u) = \ini(f_u), \ u\in N.
$$
The kernel of $\psi$ is generated by a set of binomials. In the terminology of \cite{RobSwe} a binomial in $\Ker\psi$ is called a \emph{\tat}.

A monomial in $P$ is given by an exponent vector  $e=(e_u)_{u\in N}$ of natural numbers $e_u$ of which all but finitely many are $0$. We set $Y^{e} = \prod_{u\in N} Y_u^{e_u}$. Let $F\in P$ be a polynomial, given as a $K$-linear combination of monomials, $F=\sum_i a_iY^{e_i}$ where the $e_i$ are exponent vectors, and $a_i\neq 0$ for all indices involved. We set
\begin{align*}
\ini_\phi(F)&=\max_i \  \ini(\phi(Y^{e_i})), \\
\init_\phi(F)&=\sum_{\ini(\phi(y^{e_i}))=\ini_\phi(F) } a_iY^{e_i}.
\end{align*}
Note that in the definition of $\ini_\phi(F)$ the maximum is taken over the initial monomials with respect to the monomial order on $R$ so that it is a monomial in $R$. In contrast, $\init_\phi(F)$ is a polynomial in $P$. Since $\phi(\init_\phi(F))$ can be $0$, in general $\ini_\phi(F) \neq \ini(\phi(F))$, and this cancellation of initials is the crucial  aspect of \Sagbi {} computation.  One says that a polynomial $F\in\Ker\phi$ \emph{lifts} a polynomial $H\in \Ker\psi$ if $\init_\phi(F)=\init_\phi(H)$.

We can now formulate the \emph{\Sagbi {} criterion} (see \cite[1.3.14]{BCRV}).

\begin{theorem}
With the notation introduced, et $\cB$ be a set of binomials generating $\Ker\psi$. Then the following are equivalent:
\begin{enumerate}
	\item $\cF$ is a \Sagbi {} basis of $A$;
	\item every binomial $b\in \cB$ can be lifted to a polynomial $F\in \Ker\phi$.
\end{enumerate}
\end{theorem}

The Buchberger algorithm for a Gr\"obner basis of an ideal $I$ starts from a system of generators $G$ of $I$. Then one applies two steps, namely (i) the computation of the $S$-polynomials $S(g_1,g_2)$, $g_1, g_2\in G$, and (ii) their reductions modulo $G$. The nonzero reductions are then added to $G$, and the next round of $S$-polynomials of the augmented $G$ and their reductions is run. This produces an increasing sequence of initial ideals $\ini(G)$. Because of Noetherianity the process stops after finitely many rounds with a Gr\"obner basis of $I$. The reduction of an $S$-polynomial $S(g_1,g_2)$ to $0$ is equivalent to the liftability of the  ``divided Koszul syzygy'' of $\ini(g_1)$ and $\ini(g_2)$ to a syzygy of the polynomials $g\in G$.

The computation of \Sagbi {} bases follows the same pattern. There are however two main differences: an analog of the divided Koszul syzygies does not exist, and ascending chains of monomial subalgebras of $R$ need not stabilize. For \Sagbi {} bases one must therefore compute a binomial system of generators of $\Ker\Psi$, and one cannot expect the algorithm to stop. The analog of reduction is called subduction (we again follow \cite{RobSwe}) .

\begin{definition}
Let $g\in R$. Then $r\in R$ is a \emph{\index{subduction}subduction of $g$ modulo $\cF$} if there exist monomials $\cF^{e_1},\dots,\cF^{e_m}$ and non-zero coefficients $a_i\in K$ such that the following hold:
\begin{enumerate}
	\item $g=a_1\cF^{e_1}+\dots+a_m\cF^{e_m}+r$;
	\item $\ini(\cF^{e_i})\le \ini(g)$ for $i=1,\dots.m$;
	\item no monomial $\mu\in\supp(r)$ is of type $\ini(\cF^e)$.
\end{enumerate}
\end{definition}
The process that computes a subduction of $f$ modulo $\cF$ is also called subduction. Here $\supp(r)$ is the set of monomials of $R$ appearing in $r$ with a nonzero coefficient. In the computation of \Sagbi {} bases one can replace (3) by the weaker condition
\begin{enumerate}
\item[($3'$)]  $\ini(r)$ is not of type $\ini(\cF^e)$.
\end{enumerate}
There is an obvious algorithm that produces a \emph{subduction remainder} $r$ in $(3')$: if $\ini(f)=\ini(\cF^e)$,we   replace $f$ by $f - a\phi(\cF^e)$ where $a$ is the leading coefficient of $f$, and iterate this \emph{subduction step} as long as possible. The algorithm stops since the sequence of initial monomials is descending and descending sequences in a monomial order are finite. Once $(3')$ is reached, one applies subduction steps iteratively to the remaining monomials to achieve the `tail subduction'' asked for by (3). 

The algorithm (\emph{\Sagbi {}}) starts from the finite family $\cF_0$ generating the subalgebra $A\subset R$. Then one proceeds as follows:
\begin{enumerate}
\item[(1)] Set $i=0$.
\item[(2)] Set $\cF'=\emptyset$ and compute a binomial system of generators $\cB_i$ of the kernel of $\psi_i: P_i \to K[\cF_i]$, $P_i=K[Y_F:F\in \cF_i]$, $\psi_i(Y_F)=\ini(F)$.
\item[(3)] For all $\beta\in \cB_i$ compute the subduction $r$ of $\phi_i(\beta)$ modulo $\cF_i$, $\phi_i$ given by the substitution $Y_F\mapsto F$, $F\in\cF_i$.  If $r\neq 0$, make $r$ monic and add it  to $\cF'$.
\item[(4)] If $\cF'=\emptyset$, set $\cF_j=\cF_i$, $P_j=P_i$, $\cB_j=\cB_i$ for all $j\ge i$ and stop.
\item[(5)] Otherwise set $\cF_{i+1}=\cF_i\cup \cF'$, $i=i+1$ and go to (2).
\end{enumerate}
It is not hard to see that $\cF=\bigcup_{i=0}^\infty\cF_i$ is a \Sagbi {} basis of $A$, and that the algorithm stops after finitely many steps if $A$ has a finite \Sagbi {} basis.

\begin{remark}\label{goodies}
The computation of \Sagbi {} bases is in general a very complex operation. However, in some cases it can offer  a fast solution to problems that seem much simpler at first sight, but then turn out to be very hard. We discuss two cases.

(a) In our work preparing the article \cite{BCV} with Varbaro, we did not succeed to compute the defining ideal of the algebra $A$ generated by the $2\times 2$ minors of a $4\times 4$ matrix of indeterminates over a field $K$ in ``nonexceptional'' characteristics, for example, characteristic~$0$. Singular, CoCoA and Macaulay did not stop in computing the defining ideal in reasonable time, not even up to degree $4$.

From \cite{BCPow} it was however known that $A$ has a finite \Sagbi {} basis, and meanwhile more is known about it by work of Varbaro; see \cite[6.4.10]{BCRV}. In Section \ref{comp} we come back to the computation. It takes only 40 sec; see Table \ref{Bench_1}. With the right bookkeeping one can explicitly lift the final \tat{} to a defining ideal of the algebra. Since one wants the defining ideal in terms of the original system of generators, further processing is necessary and a minimization must follow. Nevertheless this is a feasible approach.

By the work of Huang et al.\  \cite{HPPRS} it is now known that the defining ideal for algebras of $2$-minors is generated in degree $2$ and $3$.

(b) Suppose that one wants to compute the Hilbert series of a Grassmannian explicitly by a computer algebra system. It is of course possible to use representation theoretic methods or classical approaches going back to Hodge \cite{Hodge}; see Braun \cite{Braun} for explicit formulas. But these need preparations, and the same is true if one wants to exploit the Pl\"ucker relations. Computing a Gr\"obner basis by elimination is therefore tempting, but already for rather small cases it takes surprisingly long--- already $3\times 9$ takes days. In contrast, the computation from a \Sagbi {} basis given by the generating minors in a diagonal monomial order is almost instantaneous. Also the explicit computation of the \Sagbi {} basis with respect to a diagonal monomial order is very fast. See Remark \ref{OnBench_1}(f) for computational data.
\end{remark}

\section{\Sagbi {} combinatorics of maximal minors}\label{Grass}

Let $K$ be a field and $X$ an $m\times n$ matrix of indeterminates with $m\le n$. By $\cM$ we denote the set of $m$-minors of $X$, i.e., the determinants of the submatrices
$$
(X_{iu_j}: i=1,\dots,m, \ j=1,\dots,m ), \qquad u_1<\dots<u_m.
$$
The subalgebra $G(m,n) = K[\cM]$ of $R = K[X_{ij}: i=1,\dots,m,\ j=1,\dots, n]$ is an object of classical algebraic geometry, namely the homogeneous coordinate ring of the Grassmannian of $m$-dimensional vector subspaces of $K^n$ in its Pl\"ucker embedding. A ``natural'' monomial order on $R$ is lexicographic (or degree reverse lexicographic) for the order
\begin{equation}
X_{11} > \dots > X_ {1n} > X_{21} >\dots > X_{2n} >  \dots > X_{m1} >\dots > X_{mn}. \label{diag_order}
\end{equation}

It is \emph{diagonal}:  the product of the indeterminates in the diagonal is the initial monomial of each minor in $\cM$. The standard bitableaux are a $K$-basis of $G(m,n)$ (see \cite[Chap 3]{BCRV}), and this implies that $\cM$ is a \Sagbi {} basis of $G(m,n)$ for every diagonal monomial order on $R$. To the best of our knowledge, this was first observed by Sturmfels \cite{StuGrPol}. This toric deformation gives a comfortable access to the cohomological and enumerative properties of $G(m,n)$. For example see \cite[Sect. 6.2]{BCRV}.

In view of Lemma \ref{HilbLemma} it is crucial for our experimental approach to compute the Hilbert series 
$$
H_{G(m,n)}(t) = \sum_{k=0}^{\infty} (\dim_K G(m,n)_k)t^k.
$$
Usually we work with the \emph{normalized} degree on $G(m,n)$ in which the $m$-minors have degree $1$. So $G(m,n)_k$ is the degree $km$ homogeneous component of $G(m,n)$ in the standard grading of $R$. Since the Hilbert series of $G(m,n)$ and its initial algebra coincide and the initial algebra is normal, the computation of the Hilbert series by Normaliz is almost instantaneous. 

It takes some work to show that $\cM$ is a Gr\"obner basis of the ideal $I_m = I_m(X)$ of maximal minors of $X$ with respect to a diagonal monomial order. But much more is true: by a theorem of Bernstein--Sturmfels--Zelevinsky \cite{BerZel,StuZel}, $\cM$ is a \emph{universal} Gr\"obner basis of $I_m$, a Gr\"obner basis with respect to every monomial order on $R$. See \cite{CDG} for the simple proof of Conca, De Negri and Gorla and  \cite{CDG1} for further developments. This raises the question whether $\cM$ is also a universal \Sagbi {} basis for $G(m,n)$. While this is true for $m=2$ \cite[5.3.6]{BCRV}, it fails already for $m=3$, $n=6$, as observed by Speyer and Sturmfels \cite{SpeyStu}: there exists a lexicographic order on $R$ for which $\cM$ is \emph{not} a \Sagbi {} basis of $G(m,n)$. But is $\cM$ \emph{universally revlex}, namely a \Sagbi {} basis for all degree reverse lexicographic orders on $R$?

Before we answer this question, let us outline two strategies for the investigation of it and related questions. The first strategy is very simple: after fixing the matrix format, choose an order of the indeterminates, extend it to a (reverse) lexicographic order and check whether $\cM$ is a \Sagbi {} basis for this order by comparing the Hilbert functions. (Instead of varying the order, we use a random permutation of the matrix entries.) This strategy is realized in a Singular script and the Hilbert series are computed by Normaliz. In the cases in which $\cM$ is not a \Sagbi {} basis, one can furthermore try to extend it to a full \Sagbi {} basis by the algorithm documented in Section \ref{Impl}.

While the random search suggests reasonable working hypotheses, it cannot prove statements about universality for which we must exhaust all relevant monomial orders. Even if one tries to use all the symmetries of $G(m,n)$, it is impossible to scan all (reverse) lexicographic orders. Instead we start from the \emph{matching fields} in the terminology of Sturmfels-Zelevinsky \cite{StuZel}: a matching field assigns each minor a monomial in its Laplace expansion. It is called \emph{coherent} if it always selects the initial monomial with respect to a monomial order. To reduce the number of matching fields, we use a result  of Sturmfels and Zelevinsky: given  a coherent matching field, there are exactly $m(m-1)$  indeterminates that do not appear in any monomial (we call them the  ``missing" indeterminates). Sturmfels and Zelevinsky also explain that the location of these missing indeterminates  satisfies certain restrictions, for example   in each row  one finds exactly  $(m-1)$ of them.  Hence the configuration of the missing indeterminates subdivides the set of coherent matching fields. For $m=3$ and $n\ge 6$, up to row and   column permutations,  there are exactly  $4$ types. See Figure \ref{types} in which the entries $0$ denote the missing indeterminates.
\begin{figure}[hbt]
\begin{align*}
1: \begin{bmatrix}
	&0&0&\dots\\
	0&&0&\dots\\
	0&0&&\dots
\end{bmatrix}&&&
2: \begin{bmatrix}
	&0&0&&\dots\\
	0&&&0&\dots\\
	0&0&&&\dots
\end{bmatrix}\\
3:\begin{bmatrix}
	&0&0&&&\dots\\
	0&&&0&&\dots\\
	0&&&&0&\dots
\end{bmatrix}&&&
4:\begin{bmatrix}
	&&&&0&0&\dots\\
	&&0&0&&&\dots\\
	0&0&&&&&\dots
\end{bmatrix}
\end{align*}
\caption{Types of matching fields for $m=3$}\label{types}
\end{figure}

We call a matching field \emph{lex (revlex) compatible} if there exists a (reverse) lexicographic order that produces the given matching field as the sequence of the initial monomials of $\cM$. The following observation saves computation time.

\begin{proposition}
Type 4 is not lex compatible.	
\end{proposition}

\begin{proof}
Among the indeterminates one entry is largest with respect to the order, say $X_{uv}$. However, regardless of which it is, there always exist a minor involving $3$ columns such that both monomials in its Laplace expansion that are divisible by $X_{uv}$ are excluded since they both hit a $0$ in one of the rows different from row $u$.
\end{proof}

The types are scanned individually. Even after these preparations it would take too long to create all macing fields for a given type and then check  whether they are lex or revlex compatible, and if so, whether $\cM$ is a \Sagbi {} basis for such an order. Usually there are several such orders; but it depends only on the matching field whether $\cM$ is a \Sagbi {} basis. As soon as the matching field is fixed, this is only a question of whether the corresponding monomial algebra and $G(m,n)$ have the same Hilbert series.

Therefore we choose an incremental approach that is realized in a C++ program.  Let $\cM=\{\Delta_1,\dots, \Delta_N\}$, $N=\binom{n}{m}$. The matching fields under consideration are sequences of monomials $\delta_1,\dots,\delta_N$ such that $\delta_i$ appears in the Laplace expansion of $\Delta_i$ and avoids the indeterminates that are excluded by the chosen type. The notion of lex or revlex compatibility extends naturally to initial subsequences of a matching field. For a given initial subsequence $\gamma_1,\dots,\gamma_u$ let $\Gamma(\gamma_1,\dots,\gamma_u)$ be the set of all compatible matching fields that extend $\gamma_1,\dots,\gamma_u$ by $\gamma_{u+1},\dots,\gamma_N$. We want to compute $\Gamma(\emptyset)$ where $\emptyset$ stands for the empty initial subsequence. This is done via the recursive relation
$$
\Gamma(\gamma_1,\dots,\gamma_u) = \bigcup_{\gamma_{u+1}} \Gamma(\gamma_1,\dots,\gamma_u, \gamma_{u+1})
$$
where $\gamma_{u+1}$ satisfies the following conditions:
\begin{enumerate}
\item $\gamma_{u+1}$ appears in the Laplace expansion of $\Delta_{u+1}$,
\item $\gamma_{u+1}$ is not excluded by the given type, and
\item $\gamma_1,\dots,\gamma_u, \gamma_{u+1}$ is compatible.
\end{enumerate}
Less formally, we extend a compatible initial sequence $\gamma_1,\dots,\gamma_u$ in all possible ways. Condition (3) is the crucial test: if $\gamma_1,\dots,\gamma_u$ is not extensible by at least one $\gamma_{u+1}$, the recursion sops and we backtrack to $\gamma_1,\dots\gamma_{u-1}$ and try the next choice of $\gamma_u$. 

Let us state the results of the two experimental approaches:

\begin{theorem}\label{universal}
Let $m= 3$.
\begin{enumerate}
\item $\cM$ is a universal \Sagbi {} basis for $n\le5$, but not for $n\ge 6$.
\item $\cM$ is a universally revlex \Sagbi {} basis for $n\le 7$.
\item There exist a lexicographic order for $n=6$ and a reverse lexicographic order for $n=8$ such that $\cM$ is not a \Sagbi {} basis for them.
\end{enumerate}
\end{theorem}

Unexpectedly often we have observed that both $K[\ini(M)]$ and the full initial algebra $\ini(K[\cM])$ are normal:

\begin{theorem}\label{normal}
Let $m=3$.
\begin{enumerate}
\item For lex orders $K[\ini(\cM)]$ is normal for $n\le 9$.
\item For revlex orders $K[\ini(\cM)]$ is normal for $n\le 8$.
\item For $n=9$ there exists a revlex order such that $\cM$ is a \Sagbi {} basis, but $\ini(K[\cM])$ is not normal.
\item For $n = 10$ there exists a lex order such $K[\ini(\cM)]$ is not normal.
\end{enumerate}	
\end{theorem}

As a concrete example let us give a matrix for Theorem \ref{normal}(3) where $[u\mid v]$ denotes $X_{uv}$:
$$
\begin{pmatrix}
[1\mid4]& [3\mid1]& [1\mid5]& [2\mid5]& [3\mid4]& [2\mid8]& [2\mid9]& [2\mid6]& [2\mid1]\\
[1\mid9]& [1\mid3]& [2\mid2]& [2\mid7]& [1\mid6]& [3\mid5]& [3\mid2]& [1\mid1]& [2\mid3]\\
[1\mid8]& [2\mid4]& [3\mid7]& [3\mid9]& [3\mid8]& [1\mid2]& [3\mid3]& [1\mid7]& [3\mid6)
\end{pmatrix}.
$$
The indeterminates are ordered as in \eqref{diag_order} and the monomial order is the degrevlex extension; as  mentioned, we have permuted the entries of the matrix instead of changing the order of the indeterminates.

Even our recursive method for finding compatible matching fields does not reach $n=10$. Already for $n=9$ it needs weeks of computation time, whereas it finishes in hours for $n=8$. So one can say that we were rescued by the random search that produced counterexamples exactly one step beyond the limit of computability.

The experimental evidence that we have collected, makes us optimistic for the following

\begin{conjecture}
$G(m,n)$ has a finite universal \Sagbi {} basis.
\end{conjecture}

The conjecture is supported by overwhelming experimental evidence for $G(3,6)$; in fact, we expect that the universal \Sagbi {} basis has $15$ elements of (normalized) degree $2$ in addition to the $3$-minors.
We are confident that we can extend our experiments to checking the conjecture for $m = 3$ and $n\le 8$. It is already clear that $n=9$ is out of reach, not only because of the large number of cases, but also since the algorithm of Section \ref{Impl} must often give up if the degrees of the polynomials in the \Sagbi {} basis exceeds $7$, and we have seen cases in which degree $11$ is reached. Even if the combinatorial computations should be still doable, the complexity of the polynomial computations and the available memory set a limit.

\begin{remark}\label{SystRemark}
(a) While $\cM$ fails to be a \Sagbi {} basis much earlier and much more often for lex orders than for revlex ones, we have observed that the missing polynomials usually have considerably lower degree in the lex case. 

(b) Instead of lex and revlex orders one can experiment with arbitrary  orders. At least for $m= 3$ and $n\le 8$, all matching fields found for arbitrary orders are lex or revlex compatible.

(c) The algebra $G(m,n)$ is a retract of the Rees algebra $\Rees(I_m)$ by degree selection, and therefore the equality $K[\ini(\cM)] = \ini(K[\cM])$ is a necessary condition for $\Rees(I_m)=\ini(\Rees(I_m))$. Not surprisingly, it is not a sufficient condition, as many counterexamples demonstrate. Note that $\Rees(I_m)=\ini(\Rees(I_m))$ is equivalent to $\ini(I_m)^k= \ini(I_m^k)$ for all $k$. Even if this does not hold in general, it often starts to fail for unexpectedly large $k$.

As a concrete example consider the matrix
$$
\begin{pmatrix}
[2\mid 4]&[1\mid 5]&[3\mid 5]&[3\mid 6]&[1\mid 1]&[1\mid 6]&[2\mid 7]\\
[1\mid 4]&[2\mid 6]&[1\mid 7]&[3\mid 4]&[3\mid 3]&[2\mid 5]&[2\mid 3]\\
[1\mid 2]&[3\mid 7]&[3\mid 2]&[3\mid 1]&[1\mid 3]&[2\mid 1]&[2\mid 2]
\end{pmatrix}
$$
where we use the same notation and monomial order as in the example following Theorem \ref{normal}. Despite of $K[\ini(\cM)]=\ini(K[\cM])$ one has $\Rees(I_m)\neq\ini(\Rees(I_m))$, as the comparison of Hilbert series shows . Since the elements of $\cM$ have constant degree, $\Rees(I_m)$ and $\ini(\Rees(I_m))$ have a standard grading.  With CoCoA \cite{CoCoA} we have analyzed the binomial relations of $K[\ini(\cM)]$. With respect to the normalized  bigrading of the Rees algebra there are $245$ quadrics of bidegree $(1,1)$ and $(0,2)$. Those of bidegree $(1,1)$ can be lifted since the minors form a universal Gr\"obner basis. Those of bidegree $(0,2)$ have standard degree $2$, but the first degree in which the Hilbert series differ is $5$. So they are liftable as well. The obstruction to equality is a relation of bidegree $(1,4)$ and standard degree $5$. This implies that $\ini(I_3)^k=\ini(I_3^k)$ for $k\le 3$, but $\ini(I_3)^4\neq\ini(I_3^4)$.

The algorithm (Gen) of Section \ref{Impl} completes the \Sagbi {} basis by exactly one more degree $5$ element in a few seconds and confirms the analysis above: the additional element has  bidegree $(1,4)$. Consequences: (i) $\ini(I_3^4)$ and $\ini(I_3^4)$ differ in degree $13$ and (ii) $\ini(I_3^k)=\ini(I_3^4) \ini(I_3^{k-4})$ for $k\geq 4$.
\end{remark}

\section{An implementation of the \Sagbi {} algorithm}\label{Impl}

Our implementation is based on Singular \cite{Sing} and Normaliz \cite{Nmz}. We have realized three variants of the algorithm that we will explain below. They are organized in the Singular library sagbiNormaliz.lib, which in its turn connects to Normaliz for the combinatorial tasks via an extended version of the Singular library normaliz.lib. Both libraries will be published together with Normaliz version 3.10.0 (released in January  2023). 

The interface offered by normaliz.lib writes input files for Normaliz, calls it, and then reads the output files. The Macaulay2 \cite{M2}  interface normaliz.m2 \cite{BrKae} is file based as well. The transfer of the Singular implementation to Macaulay2, together with an update of normaliz.M2, is not hard for an experienced Macaulay2 user. The Normaliz team would be very grateful for help! A realization via the C++ class library libnormaliz would of course be preferable and simplify the implementation.

Independently of the \Sagbi {} computations, version 3.10.0 of Normaliz has been augmented by functions for arbitrary positive affine monoids. Such a monoid represents the combinatorial skeleton of a monomial algebra $M$, namely the monoid of the exponent vectors of the monomials of $M$. Whereas the functions are implemented additively on the combinatorial side, we describe their results in the multiplicative language of monomial algebras: 
\begin{enumerate}
\item  A minimal system of generators, which is uniquely determined. (See Bruns and Gubeladze \cite[Chap. 2]{BrGu} for the theory of affine monoids.). In fact, it is the set of \emph{irreducible} elements of $M$, i.e., those monomials $x$ that cannot be written as a product  of monomials $y,z\neq x$ . In addition, one can ask for the \emph{representations} of the reducible elements in a given system of generators as power products of the irreducible ones. 

\item A defining binomial ideal of $A$, given by a minimal system of generators. Such an ideal is often called \emph{toric} and the (minimal) system of generators is a (minimal) \emph{Markov basis}; see De Loera, Hemmecke and K\"oppe \cite{DLHK}. Contrary to (1), this is a time critical task. Normaliz has implemented the project-and-lift algorithm of Hemmecke and Malkin \cite{HemMal}. The Normaliz implementation is not inferior to Hemmecke and Malkin's 4ti2 \cite{4ti2}.

\item The Hilbert series of $M$, which is the ordinary generating function  of the enumeration of monomials by degree. It uses the classical commutation of Hilbert series via initial (monomial) ideals.
\end{enumerate}
One needs the defining binomial ideal for the \tat, as is clear by their definition. The irreducible elements must be known for a minimal \Sagbi {} basis and the control of the algorithm. The representations are necessary for subduction. One of our variants uses the Hilbert series, as we will explain below.

As far as a grading is involved, the algorithms assume the standard $\NN$-grading of the ambient polynomial ring $R$. The extension to other positive gradings would be possible without much effort, not yet however the extension to multi-gradings because the current version of Normaliz does not allow them. For a (monomial) subalgebra $A$ we work with the \emph{normalized degree}. It is obtained as follows: first one computes the greatest common divisor of the degrees of the polynomials in a generating set, and then divides the standard degree by this ``grading denominator''. This is compatible to Normaliz, which uses the normalized degree as well (unless one forbids the grading denominator).

Via the generating system $\cF = (f_u)_{u\in N}$ of the subalgebra $A$, the polynomial ring $P = K[X_u: u\in N]$ is graded as well when we set $\deg Y_u = \deg f_u$ for all $u\in N$. Under this grading all \tat{} are homogeneous.

All three variants proceed in ``rounds'' as described in the algorithm (\Sagbi {}), but are modified in two variants:
\begin{enumerate}
\item[(Gen)] The \emph{general} variant is essentially{\tiny {\tiny }} (\Sagbi {}). It stops when $\cF'$ in step (4) is empty. As this may never happen, the user can set a bound on the number of rounds. The general variant does not require a grading.

\item[(Deg)] After the computation of the \tat{} for a set $\cF_i$, the \emph{degree by degree} variant goes over the homogeneous components of the \tat{} and stops at degree $d$ as soon as at least one of the subduction remainders is nonzero and therefore at least one new element $f$ of the \Sagbi {} basis has been found. 

All subduction remainders then extend the \Sagbi {} basis to degree $d$, and the ``search degree'' can be raised to $d+1$. The stop criterion $\cF' = \emptyset$ is the same as for (Gen). It is reached in degree $d$ when no homogeneous component of the n\tat{} has a nonzero subduction remainder in degrees $>d$.

For (Deg) the user must set an upper bound for $d$. If it can be expected that the \Sagbi {} basis is finite, the upper bound should be chosen very large.

\item[(Hilb)] The \emph{Hilbert series controlled variant} refines (Deg). As input it not only needs a system of generators of the subalgebra $A$, but also the Hilbert series of $A$, given as a rational function by its numerator and denominator. 

Let $S$ be the \Sagbi {} basis of $A$ up to degree $d$. Then the Hilbert series of $K[\ini(S)]$ is computed. If it agrees with the Hilbert series of $A$, $S$ is the complete \Sagbi {} basis. Otherwise the ``critical degree'' is computed, i.e., the lowest degree in which the Hilbert functions of $K[\ini(S)]$ and $A$ differ, together with the difference.

This information not only tells us in which degree we must evaluate the \tat{} for $K[\ini(S)]$ to find the next elements of the \Sagbi {} basis, but also their number $m$. Therefore the subduction can stop as soon as $m$ nonzero pairwise different remainders have been reached. They must be irreducible since they are not divisible by monomials of smaller degree and do not divide each other.

In contrast to (Gen) and (Deg), (Hilb) offers a perfect error control, which was extremely helpful during the development of the library.
\end{enumerate}

(Gen) and (Deg) have been implemented in other packages (though not necessarily together) that will be named  in Section \ref{comp}, and only some details may vary. But we are not aware of an implementation of (Hilb).

All our variants return the (partial) \Sagbi {} basis computed and an additional integer that takes the value $0$, if the partial \Sagbi {} basis is incomplete, the value $1$ if completeness is unknown, and $2$ if the \Sagbi {} basis is complete.

\begin{remark}\label{details}
(a) Both the evaluation of the \tat{} and a subduction step can be realized by the homomorphism $\phi: P \to A$ of Equation \eqref{phi}. Both amount to mapping binomials from $P$ to $A$ via $\phi$. We use the Singular map functionality for it.

(b) For the subduction one has two choices: (i) to subduce the polynomials individually until the remainder has been found, or (ii) apply one subduction step to all polynomials simultaneously, and iterate this step as long as necessary. In our computations (ii) has proved to be the better choice.
\end{remark} 

\begin{remark}\label{choice}
The following comments will be illustrated by computational data in Section~\ref{comp}.

(a) If the given subalgebra $A$ is graded, it is always advisable to use (Deg) or even (Hilb). (Gen) should be reserved for nongraded subalgebras: a stop after a certain number of rounds gives much less information on the \Sagbi {} basis than the stop at a predefined degree. Even in graded cases in which the complete \Sagbi {} basis is computed, (Deg) is usually better.

Subduction remainders of homogeneous polynomials of degree $\le d$ with respect to a family $\cF$ of homogeneous  polynomials of degree $\le d$ remain untouched if $\cF$ is augmented by homogeneous polynomials of degree $>d$. In other words, the partial \Sagbi {} basis computed by (Deg) or (Hilb) are increasing with respect to inclusion.
 
(b) The choice between (Deg)  and (Hilb) is more difficult. (Hilb) knows exactly into which degree to look next and can often stop the subduction process much earlier than (Deg). On the other hand, the Hilbert series of $A$ must be known, and in each round the Hilbert series of a monomial algebra must be computed. 

The Hilbert series computation is very fast if the monomial algebra is normal since then the Normaliz algorithm based on triangulation and Stanley decomposition can be used. In the non-normal case it is a byproduct of the \tat{} computation that produces a Gr\"obner basis of the binomial ideal before shrinking it to a minimal system of generators, but the Hilbert series computation for the initial ideal of the \tat{} ideal does not come for free.

Another parameter is the unpredictable complexity of subduction, which in its turn depends crucially on the sparsity of the polynomials involved in it. So the decision between (Deg) and (Hilb) is a matter of the Hilbert series computation versus the subduction. 

As for (Deg) the reader must set a degree bound for (Hilb), but can ask for a final check when it is reached.

(c) In some cases the time of the \tat{} computation depends significantly on the order of the polynomials entering it (see Remark \ref{OnBench_1}(e)). Our explanation is that the Gr\"obner basis computation on which it is based depends on a monomial order and is \emph{not} invariant under the permutation of coordinates. However, we have no suggestion how to choose an order for which the \tat{} computation is especially fast. As a step in this direction, the user can ask for a sorting of the polynomials at the beginning of each round. In the implementation it is fixed to be ascending in the degrevlex monomial order on the ambient polynomial ring $R$. Note, that both (Deg) and (Hilb) generate the elements of the \Sagbi {} basis in ascending degree, but not necessarily in any more refined order.
\end{remark}

\section{Computational data}\label{comp}

In the following we give computation times for several examples and compare them to packages that are included in the computer algebra systems CoCoA5, Macaualy2 and Singular:
\begin{enumerate}
\item The computations with CoCoA 5 use a variation of the script developed by Anna Bigatti for \cite{BR} which is  available at  \cite{BB} and will be part of the   standard distribution of CoCoA from version 5.4.2. 

\item The Macaulay2 distribution, version 1.21 (December 2022)  contains the package SubalgebraBases.m2 by Burr et al., version 1.3. It realizes the variant (Deg) and allows a degree bound.

\item The Singular distribution, version 4.2.1 (May 2021) contains the library sagbi.lib by Hackfeld, Pfister and Levandovskyy. It offers only the variant (Gen) with an optional bound on the number of rounds.
\end{enumerate}
As far as we could complete the computations, all packages give the same results.

The original submission of this article and the preprint arXiv:2302.14345v1  were based on a pre-release version of Normaliz 3.10.0 and version 1.1 of SubalgebraBases.m2. For this revised version we have updated the computational data. The preprint \cite{Burr} on SubalgebraBases.m2 was not yet available at the time of our original submission.

We now list our test examples.  With the exception of (\ttt{HK}$_2$) and (\ttt{2x2$_2$}), the base field has characteristic $0$.
\begin{itemize}
\item[(\ttt{HK$_0$})] the subalgebra of $K[X,Y,Z]$ generated by the polynomials  $X^6$, $X^5Y$, $Y^5Z$, $XZ^5$,  $Y^6+ Y^3Z^3$. It is taken from Han and Kwak \cite{HanKwa} where it serves as a simple counterexample to the Eisenbud--Goto conjecture. Order is degrevlex. 

\item[(\ttt{HK$_2$})] The same, but over a field of characteristic $2$.

\item[(\ttt{Pow$_l$})] The subalgebra of $K[X,Y,Z]$ generated by the polynomials $X^6+Y^6+Z^6$, $X^7+Y^7+Z^7$, $X^8+Y^8+Z^8$. The monomial order is lex.

\item[(\ttt{Pow$_r$})] The same as (\ttt{Pow$_l$}), but order degrevlex.

\item[(\ttt{2x2$_0$})] The subalgebra of $K[X_{ij}: i,j = 1,\dots,4]$ generated by the $2$-minors. The monomial order is diagonal.

\item[(\ttt{2x2$_2$})] The same, but over a field of characteristic $2$.

\item[(\ttt{3x6})] $G(3,6)$, a nondiagonal lex order.

\item[(\ttt{3x7})]  $G(3,7)$, a nondiagonal lex order.

\item[(\ttt{3x8})] $G(3,8)$, a nondiagonal lex order.

\item[(\ttt{3x9$_l$})] $G(3,9)$, a nondiagonal lex order.

\item[(\ttt{3x9$_r$})] $G(3,9)$, a nondiagonal degrevlex order.
\end{itemize}

Table \ref{Bench_1} contains data of the examples that are relevant for the \Sagbi {} computation.
\begin{table}[hbt]
\begin{tabular}{crrrrrrrr}
\midrule[1.2pt]
\strut        & \multicolumn{3}{c}{norm deg}& & \multicolumn{3}{c}{times in minutes}\\
\cline{2-4}\cline{6-9}
\strut ecample& bound& \Sagbi {} & (Deg) & \#\Sagbi {} & (Deg) & (Hilb) & CoCoA5 & M2 \\
\midrule[1.2pt]
\strut (\ttt{HK$_0$}) & 16 & ---&  ---& 80 &1:13.67 & 1:19.56 & 3:29.00 & T\\
\hline
\strut (\ttt{HK$_2$})& 16 & ---&  ---& 16 &0:00.91 & 0:00.91  & 0:00.54 & 0:02.89 \\
\hline
\strut (\ttt{Pow$_l$}) & 200& ---& ---& 28 & 0:26.16 & 0:20.87   & 21:27.51 & ---\\
\hline
\strut (\ttt{Pow$_r$)}  & 200& --- &---& 46 & 1:24.92 & 1:10.98 & 78:58.44 & T\\
\hline
\strut (\ttt{2x2$_0$}) & 10&    3 & 7 & 89 &0:35:78 & 0:9.72 & 56:17.28   & ---\\
\hline
\strut (\ttt{2x2$_2$}) & 15&    6 & 13 &130& 2:20.30 & O      & T   & ---\\
\hline
(\strut \ttt{3x6})     & 10&    2 & 4 & 21 &0:00.45 & 0:00.41 &  0:0 .40  & 0:01.10\\
\hline
\strut (\ttt{3x7})     & 10&    2 & 4 & 37 &0:01.48 & 0:00.82 &   0:08.95 & 1:44.20\\ 
\hline
\strut (\ttt{3x8})     & 10&    3 & 6 & 67 &0:08.15 & 0:03.07 &   0:44.84 & T\\
\hline
\strut (\ttt{3x9$_l$}) & 10&    3 & 6 &101 &0:37.27 & 0:12.94 &  14:36.46 & ---\\
\hline
\strut (\ttt{3x9$_r$}) & 10&    7 & 8 & 90 &\ M & \ \ 0:22:69 &  ---       & ---\\
\hline
\end{tabular}
\vspace*{1ex}
\caption{sagbiNormaliz.lib vs. CoCoA5 and M2}	
\label{Bench_1}
\end{table}
In the table the first column after the name of the example is the degree bound for all four computations, (Deg), (Hilb), CoCoA5 and M2.The next column gives the maximum degree of the \Sagbi {} basis, provided we could compute a complete \Sagbi {} basis. The third column lists to what degree the computation had to be run for (Deg) to finish in these cases. Then we find the cardinality of the (partial) \Sagbi {} basis, followed by the computation times. Some computations failed because of error conditions or the excess of the time bound:
\begin{itemize}
\item[T] Time bound of 90 minutes exceeded. Because the time bound was exceeded for (\ttt{3x8}) by Macaulay2, we did not try larger formats.
\item[M] Lack of memory (max 32 GB), see Remark \ref{OnBench_2}(d).
\item[O] An intermediate Hilbert series could not be transferred from Normaliz to Singular because of Singular's bound of 32 bit for the type \ttt{int}.
\end{itemize}

The times are given in the format min:sec with two decimals for the seconds.  The times have been taken on a Dell xps17 laptop with an Intel  i7-11800H at 2.3 GHz, on a Dell server r6525 for Macaulay 2, and  for CoCoA5, on a MacBook Pro with an   Intel Quad-Core  i7  at 2,3 GHz.  For comparison, the r6525 have been multiplied by $0.5$, and the Macbook times have been multiplied by the factor 0.73 measured by running (\ttt{3x7}) with Macaulay2 on the three machines.

\begin{remark}\label{OnBench_1}
We add comments on specific examples.

(a) The bulk of the computation time for (\ttt{HK$_0$}) goes into the \tat{} computation which reaches rather high degrees. In contrast, the polynomials are very sparse so that the subductions are fast. This explains why (Deg) is faster than (Hilb). We have added the characteristic $2$ case since it shows that the \Sagbi {} basis depends significantly on the base field. In this case it shrinks from characteristic $0$ to characteristic $2$ (and $3$), but the opposite can happen as well.

We expect that (\ttt{HK$_0$}) and (\ttt{HK$_2$}) do not have finite \Sagbi {} bases.

(b) In contrast to (\ttt{HK$_0$}),  (\ttt{Pow$_r$}) uses more time for the subduction by rather dense polynomials. We expect that the \Sagbi {} bases are infinite.

(c) We have mentioned the $2\times 2$ minors of a $4\times 4$ matrix of indeterminates already in Remark \ref{goodies}(a). That the defining ideal was not computable for us about 10 years ago, while the \Sagbi {} basis takes only 25 sec with (Deg) and only 10 sec with (Hilb) is remarkable.

In characteristic $2$ the algebraic structure is very different from that in characteristic $0$, and this is also visible in the \Sagbi {} basis that becomes considerably larger. Since (Deg) went through we could compute the Hilbert series. But then (Hilb) failed since an intermediate Hilbert series could not be transferred because of an overflow in Singular. 

(d) The evaluation of the \tat{} on the partial \Sagbi {} basis computed can of course fail for lack of memory, as we see in (Deg) for (\ttt{3x9$_r$}). For (Deg) it is impossible to know that the \Sagbi {} basis is already complete (as is clear from (Hilb)), and degree $8$ power products of $3$-minors can already be very long polynomials. Even if the memory of a larger machine could suffice, computation time can set a limit at this point.

(e) (\ttt{HK$_0$}) has been run with sorting, all the others without. It is the only case in which we have seen a significant difference in running time. Without sorting the computation times grow by about 1 min, for (Deg) as well as for (Hilb), and the terminal output reveals that the time difference stems form the \tat{} computation.

(f) The computation of the Hilbert series of $G(3,9)$ via a \Sagbi {} basis with respect to a diagonal monomial order (not in Table \ref{Bench_1}), order takes only 12 sec.
\end{remark}

We add some computations with (Gen) to get a comparison to the Singular library sagbi.lib. For it the number of rounds of (\Sagbi {}) can be limited. We have confined ourselves to the algebras related by minors since the degrees of the \tat{} for (\ttt{HK}) and (\ttt{HK}) become extremely large after one or two rounds, and (Gen) must evaluate them fully. We have added (\ttt{3x9$_d$}), $G(3,9)$ with a diagonal monomial order.
\begin{table}[hbt]
\begin{tabular}{crrrrr}
\midrule[1.2pt]
\strut        &          \multicolumn{2}{c}{rounds}  &    \\
\cline{2-3}
\strut example     & bound& \Sagbi {} & \#\Sagbi {} & (Gen) & Singular \\
\midrule[1.2pt]
\strut (\ttt{2x2$_0$}) & 10&    3 & 89 &1:18.89 & T\\
\hline
\strut (\ttt{3x6})     & 10&    2 & 21 &0:00.58 &0:00.22\\
\hline
\strut (\ttt{3x7})     & 10&    2 & 37 &0:01.87 & F \\ 
\hline
\strut (\ttt{3x8})     & 10&    3 & 65 &0:16.47 & ---\\
\hline
\strut (\ttt{3x9$_l$}) & 10&    3 &101 &1:45.72 & ---\\
\hline
\strut (\ttt{3x9$_d$}) & 10&    1 &84 &0:11.31 & T\\
\hline
\end{tabular}
\vspace*{1ex}
\caption{sagbiNormaliz.lib vs. Singular sagbi.lib}	
\label{Bench_2}
\end{table}
In Table \ref{Bench_2} T indicates that the computation was stopped after $1$ hour without output, and F indicates a failure because of a segmentation fault in Singular.

\begin{remark}\label{OnBench_2}
As documented in Table \ref{Bench_2}, sagbi.lib computes (\ttt{3x6}).  It failed for the others. In view of the failure for (\ttt{3x7}) after $> 40$ min we did not try (\ttt{3x8}) or (\ttt{3x9$_l$}). It is certainly surprising that sagbi.lib cannot recognize that the $84$ minors form already a \Sagbi {} basis for (\ttt{3x9$_d$}). 
\end{remark}


\begin{thebibliography}{15}

\bibitem{4ti2} 
4ti2 team. 4ti2-A software package
for algebraic, geometric and combinatorial problems on linear
spaces. Available at \url{https://github.com/4ti2/4ti2}.

\bibitem{CoCoA}
John Abbott, Anna Maria Bigatti, Lorenzo Robbiano. 
CoCoA: a system for doing Computations in Commutative Algebra.
Available at \url{http://cocoa.dima.unige.it} 
 

\bibitem{BerZel}
D. Bernstein and A. Zelevinsky. Combinatorics of maximal minors. J. Algebraic Combin.
2, 2 (1993), 111--121.

\bibitem{BR} A. Bigatti and L. Robbiano. Saturations of subalgebras, SAGBI bases, and U-invariants. J. Symbolic Comput. 109 (2022), 259--282.

\bibitem{BB} A. Bigatti, A CoCoA5 package for \Sagbi {} bases, \\ Available at https://www.dima.unige.it/~bigatti/data/ComputingSaturationsOfSubalgebras/ 

\bibitem{Braun}
L. Braun.
Hilbert series of the Grassmannian and k-Narayana numbers.
Commun. Math. 27 (2019), no. 1, 27--41.

\bibitem{BCPow}
W. Bruns and A. Conca. KRS and powers of determinantal ideals. Compositio Math. 111,
1 (1998), 111--122.

\bibitem{BCV}
W. Bruns, A. Conca and M. Varbaro. Relations between the minors of a generic matrix.
Adv. Math. 244 (2013), 17--206.

\bibitem{BCRV}
W. Bruns, A. Conca, C. Raicu and M. Varbaro. Determinants, Gr\"obner bases and cohomology.
Springer Monographs in Mathematics, Springer, Cham 2022.

\bibitem{BrGu}
W. Bruns and J. Gubeladze. Polytopes, rings, and $K$-theory. Springer Monographs in
Mathematics, Springer, Dordrecht, 2009.

\bibitem{BH}
W. Bruns and J. Herzog. Cohen-Macaulay rings.Rev. ed. Cambridge: Cambridge University
Press, 1998.

\bibitem{BrKae}
W. Bruns and G. K\"ampf.
A Macaulay2 interface for Normaliz.
J. Softw. Algebra Geom. 2 (2010), 15--19.


\bibitem{Nmz}
W.~Bruns, B.~Ichim, C.~S\"oger and U.~von~der~Ohe,	Normaliz. Algorithms for rational cones and affine monoids. Available at \url{https://normaliz.uos.de}.

\bibitem{Burr} M. Burr, O. Clarke, T. Duff, J. Leaman, N. Nichols and E. Walker. SubalgebraBases in Macaulay2. Preprint \url{https://arxiv.org/abs/2302.12473}.


\bibitem{CDG}
A. Conca, E. De Negri, and E. Gorla. Universal Gr\"obner bases for maximal minors. Int. Math.
Res. Not. 11 (2015), 3245--3262.

\bibitem{CDG1}
A. Conca, E. De Negri, and E. Gorla.  Universal Gr\"obner bases and Cartwright-Sturmfels ideals. Int. Math. Res. Not.  7 (2020), 1979--1991.  

\bibitem{CHV}
A. Conca, J. Herzog, and G. Valla. \Sagbi {} bases with applications to blow-up algebras. J.
Reine Angew. Math. 474 (1996), 113--138.

\bibitem{CM} 
 O. Clarke  and F. Mohammadi,  
Toric degenerations of Grassmannians and Schubert varieties from matching field tableaux.  
J. Algebra 559 (2020), 646--678. 


\bibitem{Sing}	
W. Decker, G.-M. Greuel, G. Pfister, and H. Schonemann. Singular 4-1-1 — A
computer algebra system for polynomial computations. \url{http://www.singular.uni-kl.de}, 2018.

\bibitem{DLHK}	
J. A. De Loera, R. Hemmecke and M. K\"oppe.
Algebraic and geometric ideas in the theory of discrete optimization.
MOS-SIAM Series on Optimization, 14. Society for Industrial and Applied Mathematics (SIAM), Philadelphia 2013.


\bibitem{EneHerz}
V. Ene and J. Herzog. Gr\"obner bases in commutative algebra, vol. 130 of Graduate Studies
in Mathematics. American Mathematical Society, Providence, RI, 2012.



\bibitem{M2}
D. R. Grayson and M. E. Stillman. Macaulay2, a software system for research in algebraic
geometry. Available at \url{http://www.math.uiuc.edu/Macaulay2/}.

\bibitem{HanKwa}
J. I. Han and S. Kwak. Projective surfaces in $\mathbb{P}^4$ that are counterexamples to the Eisenbud-Goto regularity conjecture. Preprint \url{https://arxiv.org/abs/2210.07174}.

\bibitem{HemMal}
Hemmecke and P. N. Malkin. Computing generating sets of lattice ideals and Markov bases of
lattices. J. Symb. Comp. 44, 1463--1476 (2009).

\bibitem{HO} A. Higashitani, H. Ohsugi. Quadratic Gr\"obner bases of block diagonal matching field ideals and toric degenerations of Grassmannians. J. Pure Appl. Algebra 226 (2022), no. 2. 


\bibitem{Hodge}
W.V.D.Hodge. Some enumerative results in the theory of forms. Proc.Cambridge Phil.Soc. 39,24--26 (1943).

\bibitem{HPPRS}
H. Huang, M. Perlman, C. Polini, C. Raicu, and A. Sammartano. Relations between the
$2\times 2$ minors of a generic matrix. Adv. Math. 386 (2021).


\bibitem{KapMad}
D. Kapur and K. Madlener. A completion procedure for computing a canonical basis for a
$k$-subalgebra. In Computers and mathematics (Cambridge, MA, 1989). Springer, New York,
1989, pp. 1--11.

\bibitem{KrRo}
M. Kreuzer and L. Robbiano. Computational commutative algebra II. Springer-Verlag,
Berlin, 2005.


\bibitem{MilStu}
E. Miller and B. Sturmfels. Combinatorial commutative algebra, vol. 227 of Graduate
Texts in Mathematics. Springer-Verlag, New York, 2005.


\bibitem{RobSwe}
L. Robbiano and M. Sweedler. Subalgebra bases. In Commutative algebra (Salvador,
1988), vol. 1430 of Lecture Notes in Math. Springer, Berlin, 1990, pp. 61--87.

\bibitem{SpeyStu}
D. Speyer and B. Sturmfels. The tropical Grassmannian. Adv. Geom. 4, 3 (2004), 389--411.

\bibitem{StuGrDet}
B. Sturmfels. Gr\"obner bases and Stanley decompositions of determinantal rings. Math. Z.
205, 1 (1990), 13--144.

\bibitem{StuGrPol}
B. Sturmfels. Gr\"obner bases and convex polytopes, vol. 8 of University Lecture Series.
American Mathematical Society, Providence, RI, 1996.

\bibitem{StuZel}
B. Sturmfels and A. Zelevinsky. Maximal minors and their leading terms. Adv. Math. 98,
1 (1993), 65--112.


\end{thebibliography}
\end{document}